\documentclass[hidelinks, 11pt]{article}
\usepackage[a4paper,bindingoffset=0.in,left=2.5cm,right=2.5cm,top=2cm,bottom=3cm]{geometry}
\usepackage{amsmath, amssymb,amsthm}
\usepackage{graphicx,subfigure}
\usepackage{url}
\usepackage{enumitem}
\usepackage[utf8]{inputenc}
\usepackage{lmodern}
\usepackage{setspace}
\usepackage[normalem]{ulem}

\usepackage{wrapfig}
\usepackage{mdframed}
\usepackage[a]{esvect}
\usepackage{color}

\usepackage{epstopdf} 

\newtheorem{theorem}{Theorem}[section]
\newtheorem{proposition}[theorem]{Proposition}

\newtheorem{lemma}[theorem]{Lemma}

\newtheorem{definition}[theorem]{Definition}

\newtheorem{corollary}[theorem]{Corollary}
\newtheorem{conjecture}[theorem]{Conjecture}

\theoremstyle{definition} 
\newtheorem{definition}[theorem]{Definition}
\newtheorem{remark}[theorem]{Remark}
\newtheorem{example}[theorem]{Example}
\newtheorem{question}[theorem]{Question}

\usepackage{mathtools, amsmath, amssymb, graphicx, amsfonts}
\usepackage{xcolor} 
\usepackage[colorlinks=true,allcolors=blue!75!black]{hyperref} 
\usepackage{bbm} 
\usepackage{float} 
\usepackage{amsopn} 
\usepackage{mathrsfs} 
\usepackage{comment} 
\usepackage{hyperref} 
\usepackage{amsthm} 
\renewenvironment{proof}{{\noindent\bfseries Proof.}}{\qed} 

\newtheorem{innercustomthm}{Theorem}
\newenvironment{customthm}[1]
  {\renewcommand\theinnercustomthm{#1}\innercustomthm}
  {\endinnercustomthm}
\newtheorem{innercustomprop}{Proposition}
\newenvironment{customprop}[1]
  {\renewcommand\theinnercustomprop{#1}\innercustomprop}
  {\endinnercustomprop}

\renewcommand{\ker}{\operatorname{ker}}

\newcommand{\spn}{\operatorname{span}}

\newcommand{\conv}{\operatorname{conv}}

\newcommand{\R}{\mathbb{R}}

\newcommand{\N}{\mathbb{N}}
\newcommand{\Z}{\mathbb{Z}}
\newcommand{\GLnsym}{\operatorname{GL}_n^{\text{sym}}}
\newcommand{\kron}{\operatorname{kron}}
\newcommand{\inv}{\operatorname{cinv}}

\usepackage{tikz}
\newsavebox{\mydiamondbox}
\sbox{\mydiamondbox}{%
  \tikz[scale=0.08, baseline=-0ex, line width=0.5pt]{
    \draw (0,1) -- (1,2) -- (2,1) -- (1,0) -- cycle; 
    \draw (0,1) -- (2,1); 
  }%
}

\newcommand{\hdiamond}{\mathrel{\usebox{\mydiamondbox}}}

\begin{document}

\title{Graphs with nonnegative resistance curvature}

\author{Karel Devriendt\\\small{University of Oxford, UK}}

\date{}

\maketitle
\begin{abstract}
This article introduces and studies a new class of graphs motivated by discrete curvature. We call a graph \emph{resistance nonnegative} if there exists a distribution on its spanning trees such that every vertex has expected degree at most two in a random spanning tree; these are precisely the graphs that admit a metric with nonnegative resistance curvature, a discrete curvature introduced by Devriendt and Lambiotte. We show that this class of graphs lies between Hamiltonian and $1$-tough graphs and, surprisingly, that a graph is resistance nonnegative if and only if its twice-dilated matching polytope intersects the interior of its spanning tree polytope. We study further characterizations and basic properties of resistance nonnegative graphs and pose several questions for future research.
\end{abstract}


\section{Introduction}
\subsection{Motivation}
In \cite{devriendt_2022_curvature}, the author and Lambiotte introduced a new notion of discrete scalar curvature on the vertices of a graph, based on the geometry of effective resistances. For a simple graph $G=(V,E)$ with edge weights $c>0$, the \emph{effective resistance} $\omega$ between two adjacent vertices is proportional to the $c$-weighted fraction of all spanning trees that contain the edge between them: 
\begin{equation}\label{eq: definition effective resistance}
\omega_{uv} = c_{uv}^{-1}\frac{\sum_{T\ni uv}\prod_{t\in T}c_t}{\sum_{T}\prod_{t\in T}c_t}  \quad\quad\text{~for $uv\in E$,}
\end{equation}
where $T$ runs over all spanning trees in $G$. The effective resistance originates from electrical circuit analysis and potential theory and it has a rich combinatorial, algebraic and geometric theory; see for instance \cite{biggs_1993_algebraic, devriendt_2022_thesis, fiedler_2011_matrices}. The \emph{resistance curvature} $p$ is a discrete scalar curvature on the vertices of a graph, defined in \cite{devriendt_2022_curvature} as:
$$
p_v(c) = 1 - \frac{1}{2}\sum_{uv\in E}\omega_{uv}c_{uv}  \quad\quad\text{~for $v\in V$}.
$$
As in differential geometry, graphs with nonnegative resistance curvature have a lot of nice properties \cite{devriendt_2022_thesis, devriendt_2022_curvature, devriendt_2024_graph} and we note in particular a strong structural theorem for positively curved graphs; a graph $G$ is called \emph{$t$-tough} if for every set of vertices $U\subset V$ whose removal disconnects the graph, the graph $G-U$ has at most $\vert U\vert/t$ components. The following is implied by a result of Fiedler in simplex geometry \cite[Thm. 3.4.18]{fiedler_2011_matrices} and another proof is given in \cite[Thm. 6.31]{devriendt_2022_thesis}:
\begin{theorem}[\cite{fiedler_2011_matrices}]\label{thm: toughness}
Connected weighted graphs with positive resistance curvature are $1$-tough.
\end{theorem}
In this result, we point out that positive resistance curvature depends on a choice of edge weights $c$, while $1$-tough is a purely combinatorial property that only depends on the graph data $G=(V,E)$. Thus, if one can find some edge weights for which $p>0$, then Theorem \ref{thm: toughness} applies and the graph is 1-tough. This observation naturally motives the following question:
\begin{question}\label{q: main question}
For which graphs $G$ does there exist a choice of edge weights $c$ such that $p\geq 0$, respectively $p>0$? We call these graphs \emph{resistance nonnegative} (RN) and \emph{resistance positive} (RP). A graph which is RN but not RP is called \emph{strictly resistance nonnegative} (SRN).
\end{question}
~\vspace{-2em}
\begin{figure}[h!]
    \centering
    \includegraphics[width=0.7\linewidth]{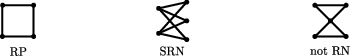}
    \caption{An RP, SRN and not RN graph. See Examples \ref{ex: RP graph}--\ref{ex: not RN graph} for details.}
    \label{fig: first examples}
\end{figure}

Geometrically, by thinking of edge weights as inverse edge lengths, this question can be reformulated as ``which graphs admit a metric with nonnegative/positive scalar curvature". Replacing ``graph" by ``manifold", this parallels a classical question in differential geometry \cite{schoen_1987_structure}. With the definitions in Question \ref{q: main question}, we can turn Theorem \ref{thm: toughness} into a purely structural result.
\begin{theorem}\label{thm: RP implies 1-tough}
Connected resistance positive graphs are 1-tough.
\end{theorem}
In this article, we will answer Question \ref{q: main question} and study characterizations and properties of RN, RP and SRN graphs. We formulate several questions for future research. 

\subsection{Main results}
As a first main result, we show that RN, RP and SRN graphs can be characterized in terms of the existence of certain distributions on their spanning trees; a distribution is \emph{positive} if every spanning tree is assigned a non-zero probability.
\begin{customthm}{\ref{th: RN and expected degree}}\label{thm: distribution characterization}
A graph $G$ is RN (resp. RP) if and only if there exists a positive distribution on the set of spanning trees of $G$, such that the expected degree of every vertex in a random spanning tree under this distribution is at most $2$ (resp. strictly smaller than $2$).
\end{customthm}
The actual statement of Theorem \ref{thm: distribution characterization} in the main text is stronger and allows for other types of distributions. Our proof makes use of the fact that $(c_{uv}\omega_{uv})_{uv\in E}$ are marginal probabilities of a so-called log-linear distribution on spanning trees induced by $c$. By classical results in statistics on log-linear distributions, as one ranges over all possible weights $0<c<\infty$ these marginal probabilities trace out the (interior of the) set of marginals of all possible distributions on spanning trees. This set is a polytope $P(G)\subset \R^E$ and is called the \emph{spanning tree polytope} of $G$. Since the expected degree constraints are affine, a practical consequence of Theorem \ref{thm: distribution characterization} is that computing the RN property of a graph is a linear program on $P(G)$.
\\~\\
As a second main result, we give a geometric characterization of RN graphs via the intersection of two polytopes associated to a graph; recall that a matching is a subset of edges $M\subseteq E$ which covers every vertex at most once. The spanning tree polytope and \emph{matching polytope} of $G$ are
$$
P(G) :=\operatorname{conv}(e_T\mid T\text{~is a spanning tree})\subset \R^E, \quad M(G) := \operatorname{conv}(e_M\mid M\text{~is a matching})\subset \R^E,
$$
where $e_T=\sum_{i\in T}e_i$ with $e_i$ basis vectors of $\R^E$. We use the notation $2M(G)=\{2x\mid x\in M(G)\}$, and $P(G)^{\circ}$ for the relative interior of the spanning tree polytope (see Remark \ref{rem: open spanning tree polytope}).
\begin{customthm}{\ref{thm: intersecting trees and matchings}}
A graph $G$ is RN if and only if $P(G)^{\circ}\cap2M(G)$ is non-empty.
\end{customthm}
Following this theorem, we initiate a study of the polytope $\Theta(G)=P(G)\cap 2M(G)$.
\begin{customprop}{\ref{prop: integer points are hamiltonian paths}}
The integer points of $\Theta(G)$ are indicator vectors of Hamiltonian paths in $G$.
\end{customprop}
As a direction for future research, we propose to study Ehrhart-type questions for this polytope: what can be said about the integer points of $k\Theta(G)$ for $k\in\N$? Combinatorially, this requires counting arrangements of double matchings that give arrangements of spanning trees.
\\~\\
As a third line of results, we study some operations on graphs and their effect on resistance nonnegativity. We show that adding edges preserves RP, that `Kron reduction' (Definition \ref{def: definition Kron reduction}) preserves both RP and RN but not SRN and that `circle inversion' (Definition \ref{def: definition circle inversion}) preserves RN. One application of these closure results is the following theorem:
\begin{customthm}{\ref{thm: Hamiltonian implies RP}}
Hamiltonian graphs are resistance positive, but the converse is not always true.
\end{customthm}
Graph toughness was introduced by Chv\'{a}tal to study Hamiltonian graphs and $1$-tough does not imply Hamiltonicity; see \cite{bauer_2006_toughness}. We ask whether the class of RP graphs lies strictly between Hamiltonian and $1$-tough graphs and propose a relaxation of a conjecture of Chv\'{a}tal.
\\~\\
Finally, we discuss two further characterizations of RN and RP graphs. Similar to encoding the RN conditions for a graph $G$ by affine inequalities on the space $P(G)\subset \R^E$ that parametrizes all effective resistances in a graph, we introduce two other parametrization spaces $\mathcal{K}(G)\subset \R^{V\times V}$ and $\mathcal{S}(G)\subset [0,1]^{2^V}$ on which the RN conditions are affine inequalities. Here, $\mathcal{K}(G)$ is the set of all inverse resistance matrices of a graph and $\mathcal{S}(G)$ the space of all resistance capacities. The latter are certain set function on $2^V$ induced by effective resistances. The conjectured inequalities on $\mathcal{S}(G)$ are the submodularity inequalities, which are known to be necessary for RN.
\\~\\
\textbf{Organization:} In Section \ref{sec: relative resistances and spanning tree polytope} we discuss effective resistances, their relation to distributions on spanning trees, and the spanning tree polytope $P(G)$. In Section \ref{sec: resistance nonnegative graphs}, we define RN, RP and SRN graphs, give some examples and prove a first characterization result, Theorem \ref{th: RN and expected degree}. Section \ref{sec: trees and double matchings} deals with the second characterization result, Theorem \ref{thm: intersecting trees and matchings}, which relates resistance nonnegativity and matchings. We initiate the study of the polytope $\Theta(G)$. In Section \ref{sec: algebraic characterization} we prove a third characterization result, Corollary \ref{cor: RN from inverse resistance matrices}, in terms of the space of inverse resistance matrices $\mathcal{K}(G)$. Section \ref{sec: graph operations} deals with three graph operations and their influence on resistance nonnegativity. Lastly, Section \ref{sec: nonnegativity and submodularity} conjectures a fourth characterization, Conjecture \ref{conj: RN from submodularity}, which relates resistance nonnegativity to submodularity of the resistance capacity.

\section{Relative resistance \& the spanning tree polytope}\label{sec: relative resistances and spanning tree polytope}
Throughout the article, we will consider finite graphs $G=(V,E)$ which are simple (no multi-edges), loopless (no $vv$ edges) and connected, unless stated otherwise. A graph is \emph{biconnected} if $G-v$ is connected for every $v\in V$. Many of the presented results will hold in a more general setting, but this is left for later development. The following definition is central in this article:
\begin{definition}[spanning tree]
A \emph{spanning tree} of $G$ is a maximal acyclic subset of edges.
\end{definition}
In other words, a spanning tree is a set of edges $T\subseteq E$, such that $T$ contains no cycles in $G$ but adding any further edge from $E\setminus T$ creates a cycle. Every spanning tree in a connected graph has $\vert V\vert-1$ edges. We will use the fact that spanning trees form a matroid \cite[Ch. 13]{michalek_2021_invitation}.
\begin{definition}[spanning tree polytope]
The \emph{spanning tree polytope} of a graph $G$ is
$$
P(G) = \operatorname{conv}(e_T\mid T\text{~is a spanning tree})\subset \R^E.
$$
\end{definition} 
\begin{example}
The $3$-cycle has vertices $u,v,w$ and edges $uv,vw,uw$. It has $3$ spanning trees $\{uv,vw\}, \{vw,uw\},\{uv,uw\}$ and $P(G)$ is a triangle in $\R^3$ with vertices $(1,1,0),(0,1,1),(1,0,1)$.
\end{example}
\begin{example}[diamond graph]\label{ex: diamond graph}
Figure \ref{fig: spanning tree polytope} shows the $8$ spanning trees of the diamond graph $G_{\hdiamond}=K_4-e$, arranged according to the $1$-skeleton of the spanning tree polytope $P(G_{\hdiamond})\subset \R^5$.
\begin{figure}[h!]
    \centering
    \includegraphics[width=0.5\linewidth]{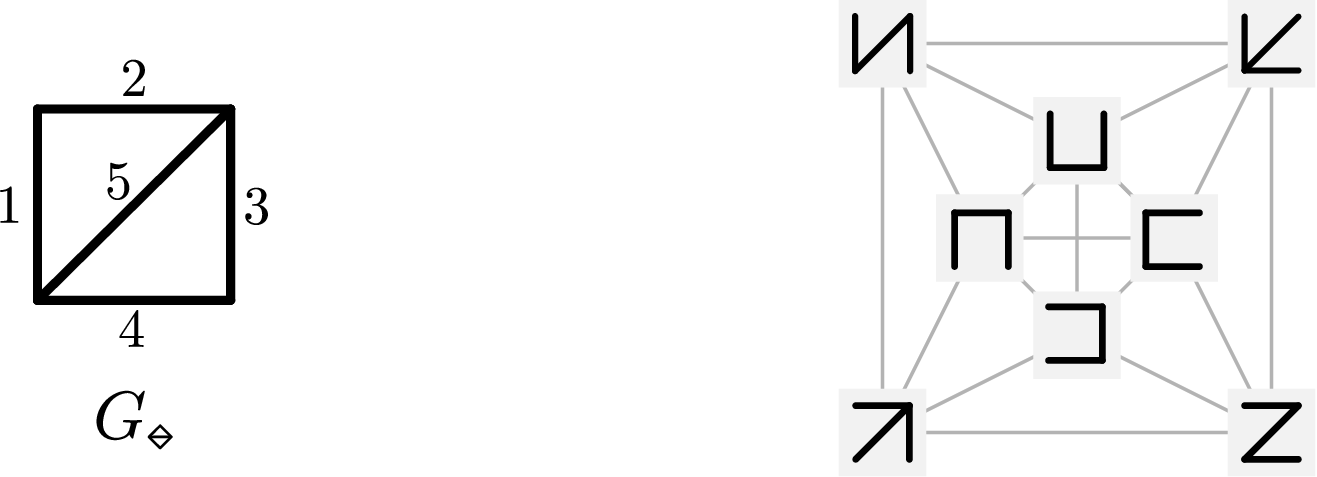}
    \caption{The diamond graph $G_{\hdiamond}$ with edge labels and the $1$-skeleton of its spanning tree polytope $P(G_{\hdiamond}$). The $8$ vertices of $P(G_{\hdiamond})$ are indicator vectors of     the spanning trees of $G_{\hdiamond}$ and the edges of $P(G_{\hdiamond})$ correspond to pairs of spanning trees that differ by two edges.}
    \label{fig: spanning tree polytope}
\end{figure}
\end{example}
\begin{remark}[open spanning tree polytope]\label{rem: open spanning tree polytope}
The spanning tree polytope is not full-dimensional in $\R^E$ by Foster's theorems (see below). For this reason, when referring to the \emph{open spanning tree polytope} $P(G)^{\circ}$, we mean its relative interior in the affine subspace in which it lives. In particular, the spanning tree polytope of a tree $T$ is a point $e_T$, and so is its relative interior. This is important when $G$ is a path, for instance in Theorem \ref{thm: intersecting trees and matchings}.
\end{remark}
In a weighted graph, one associates finite positive weights $(c_e)_{e\in E}$ to the edges of $G$. Rather than working with the effective resistance $\omega_{uv}$ as introduced in \eqref{eq: definition effective resistance}, we consider the \emph{relative resistance} $r_e=\omega_{e}c_{e}$ defined on the edges of a graph. We recall its definition.
\begin{definition}[relative resistance]
The \emph{relative resistance} of an edge $e\in E$ is
\begin{equation}\label{eq: definition relative resistance}
r_e(c) = \frac{\sum_{T\ni e}\prod_{t\in T}c_t}{\sum_{T}\prod_{t\in T}c_t}.
\end{equation}
\end{definition}
The homogeneous polynomial $Z_G(c)=\sum_T\prod_{t\in T}c_t$ in the denominator of \eqref{eq: definition relative resistance} is called the Kirchhoff polynomial of $G$, and we note that $r_e\cdot Z_G = c_e\cdot\partial Z_G/\partial c_e $. From Euler's homogeneous function theorem, one recovers the following classical result in electrical circuit analysis:
\begin{theorem}[Foster's theorem \cite{foster_1949_average}]
$\sum_{e\in E}r_e(c) = \vert V\vert -1$ for any $0<c<\infty$.
\end{theorem}
We strengthen Foster's theorem to a localized version on the biconnected components of $G$:
\begin{theorem}[local Foster's theorem]
$\sum_{e\in E(U)}r_e(c) = \vert U\vert -1$ for any $0<c<\infty$, whenever $U\subseteq V$ and $E(U)\subseteq E$ are the vertices and edges of a biconnected component of $G$.
\end{theorem}
\begin{proof}
This follows from \eqref{eq: definition relative resistance} and the fact that every spanning tree has the same number of edges in every biconnected component, namely, the size of the component minus one. The biconnected components of a graph are the connected components of the corresponding cycle matroid. 
\end{proof}

Next, we consider the subset of $\R^E$ traced out by $(r_e(c))_{e\in E}$ as we range over all possible weights $c$. This will turn out to be precisely the interior of the polytope $P(G)$. This generalizes Foster's theorem, since it characterizes all possible tuples of relative resistances by a collection of affine equalities and inequalities in $r$. Our discussion will follow a probabilistic approach.
\\~\\
Let $\mathcal{T}=\mathcal{T}(G)$ be the collection of all spanning trees of $G$. Then $\mu_c:\mathcal{T}\rightarrow\R$ given by 
$$
\mu_c(T) = \frac{\prod_{t\in T}c_t}{Z_G(c)}>0 \quad\quad\text{~for $T\in\mathcal{T}$,}
$$
is a distribution on spanning trees, and we call $\mu_c(T)$ the probability of $T$. Since this distribution is of the form $\mu_c(T)\propto\operatorname{exp}(\sum_{t\in T}y_t)$ under the transformation $y_t=\log c_t$, this is also called a \emph{log-linear distribution}; see for instance \cite{sullivant_2024_algebraic}. We consider two further types of distributions on $\mathcal{T}$, with increasing generality. First, a \emph{positive distribution} assigns a positive probability to each spanning tree. For the second type, we need the following class of subsets of spanning trees:
\begin{proposition}\label{prop: non-separable trees}
The following are equivalent for a subset of spanning trees $\mathcal{F}\subseteq\mathcal{T}(G)$:
\begin{enumerate}
	\item $\conv(e_T\mid T\in\mathcal{F})$ intersects $P(G)^{\circ}$;
	\item $\conv(e_T\mid T\in\mathcal{F})$ is not contained in a proper face of $P(G)$;
	\item $\min(w(T)\mid T\in\mathcal{F})<\max(w(T)\mid T\in\mathcal{T})$ for all $w\in\R^E$,
\end{enumerate}
where $w(T)=\sum_{t\in T}w_t$ and $w$ is not piecewise constant on the biconnected components of $G$. We call these sets \emph{non-separable}; this is not standard terminology.
\end{proposition}

\begin{proof}
Equivalence of the first two statements follows immediately from the fact that $(e_T)_{T\in\mathcal{F}}$ is a subset of the vertices of $P(G)$ and that proper faces of a polytope have positive codimension. The third statement says that there is no linear function $\ell_w: x\mapsto \langle w,x\rangle$ on $P(G)$ which is maximized on the whole set $(e_T)_{T\in\mathcal{F}}$. This is equivalent to saying that these points are not contained in a face of $P(G)$, which is statement 2. By Foster's local theorem, it follows that $\ell_w$ is constant on the vertices of $P(G)$ when $w$ is piecewise constant on biconnected components.
\end{proof}

A distribution on spanning trees is called \emph{non-separable} if it is supported on a non-separable subset of spanning trees. We note that {non-separable $\supset$ positive $\supset$ log-linear}. Finally, the tuple $(\sum_{T\ni e}\mu(T))_{e\in E} \in \R^E$ are called the \emph{edge-marginals} of the distribution $\mu$. These are the probabilities that a random spanning tree following $\mu$ contains each of the respective edges.
\begin{theorem}\label{thm: relative resistances and P(G)}
Let $G$ be a graph, then the following subsets of $\R^E$ coincide:
\begin{enumerate}
    \item the open spanning tree polytope $P(G)^{\circ}$;
    \item edge marginals of non-separable distributions on $\mathcal{T}(G)$;
    \item edge marginals of positive distributions on $\mathcal{T}(G)$;
    \item edge marginals of log-linear distributions on $\mathcal{T}(G)$;
    \item relative resistances of the edges in $G$, over all possible edge weights.
\end{enumerate}
Every point in this subset corresponds to a unique choice of edge weights, up to common rescaling in the biconnected components of $G$.
\end{theorem}
\begin{proof}
The translation between probability and geometry goes via interpreting the probabilities $\mu(T)$ as convex coefficients of the vertices $e_T$. The edge marginals of a distribution are then precisely the coordinates of the point in $P(G)$ determined by these convex coefficients:
$$
\Big(\sum_{T}\mu(T)e_T\Big)_i\,=\,\sum_{T}\mu(T)(e_T)_i \,=\, \sum_{T\ni i}\mu(T)\quad\quad\text{~for $i\in E$}.
$$
With this translation, statements 1., 2. and 3. are equivalent by definition of non-separable and positive distributions. By the inclusion {positive $\supset$ log-linear} we have $4.\Rightarrow 3.$ and by definition of relative resistances, we have $5.\Leftrightarrow 4.$ It remains to prove $3. \Rightarrow 4.$ and the correspondence of points in $P(G)^{\circ}$ with unique weights.

Birch's theorem in statistics \cite[Cor. 8.25]{michalek_2021_invitation} says that there is a unique log-linear distribution $\mu_c$ among all positive distributions with the same edge marginals. This is precisely $3.\Rightarrow 4.$

Finally, the only way in which different choices of edge weights $c,\tilde{c}$ lead to the same log-linear distribution $\mu_c=\mu_{\tilde{c}}$ is if they differ by a multiplicative factor in each biconnected component. This follows from the fact that all spanning trees $T$ contain equally many edges from each of the biconnected components, and that these are the only edge subsets with this property. Thus, by uniqueness from Birch's theorem, every point $r\in P(G)^{\circ}$ corresponds to a unique choice of finite positive edge weights $c$, up to rescaling in the biconnected components of $G$.
\end{proof}

We illustrate the equivalences in Theorem \ref{thm: relative resistances and P(G)} for the diamond graph (Example \ref{ex: diamond graph}).
\begin{example}
The open spanning tree polytope $P(G_{\hdiamond})^{\circ}$ contains the point $(\tfrac{3}{5},\dots,\tfrac{3}{5})$. We give three distributions $\mu_1,\mu_2,\mu_3$ on $\mathcal{T}(G_{\hdiamond})$ that have this point as edge marginal. First, let
$$
\mu_1(\raisebox{-0.2em}{\includegraphics[width=0.02\textwidth]{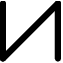}}) 
= \mu_1(\raisebox{-0.2em}{\includegraphics[width=0.02\textwidth]{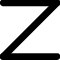}}) 
= \tfrac{3}{10}
\quad\text{~and~}\quad
\mu_1(\raisebox{-0.2em}{\includegraphics[width=0.02\textwidth]{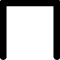}}) 
= \mu_1(\raisebox{-0.2em}{\includegraphics[width=0.02\textwidth]{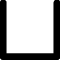}})
= \mu_1(\raisebox{-0.2em}{\includegraphics[width=0.02\textwidth]{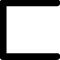}})
= \mu_1(\raisebox{-0.2em}{\includegraphics[width=0.02\textwidth]{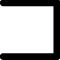}})
= \tfrac{1}{10} \quad\text{~and\quad$\mu_1=0$ otherwise.}
$$
The distribution $\mu_1$ is non-separable but not positive (e.g. $\mu_1(\raisebox{-0.2em}{\includegraphics[width=0.02\textwidth]{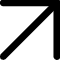}})=0$). Second, let
$$
\mu_2(\raisebox{-0.2em}{\includegraphics[width=0.02\textwidth]{tree1.eps}})
=
\mu_2(\raisebox{-0.2em}{\includegraphics[width=0.02\textwidth]{tree2.eps}})
=
\mu_2(\raisebox{-0.2em}{\includegraphics[width=0.02\textwidth]{tree7.eps}})
=
\mu_2(\raisebox{-0.2em}{\includegraphics[width=0.02\textwidth]{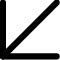}})
= \tfrac{3}{20}
\quad\text{~and~}\quad 
\mu_2(\raisebox{-0.2em}{\includegraphics[width=0.02\textwidth]{tree3.eps}})
=
\mu_2(\raisebox{-0.2em}{\includegraphics[width=0.02\textwidth]{tree4.eps}})
=
\mu_2(\raisebox{-0.2em}{\includegraphics[width=0.02\textwidth]{tree5.eps}})
= 
\mu_2(\raisebox{-0.2em}{\includegraphics[width=0.02\textwidth]{tree6.eps}})
= \tfrac{2}{20}.
$$ 
The distribution $\mu_2$ is a positive, non-separable and log-linear distribution induced by edge weights $c=(2,2,2,2,3)$, following the edge labeling in Figure \ref{fig: spanning tree polytope}. Third, the distribution $\mu_3=\tfrac{1}{2}\mu_1+\tfrac{1}{2}\mu_2$ is positive and non-separable but not log-linear. Finally, we note that the relative resistances in $G_{\hdiamond}$ with edge weights $c=(2,2,2,2,3)$ are equal to $r(c)=(\tfrac{3}{5},\dots,\tfrac{3}{5})$.
\end{example}

\begin{remark}
Following Theorem \ref{thm: relative resistances and P(G)}, relative resistances $r\in P(G)^{\circ}$ can be used as an alternative way to parametrize edge weights on graphs. The theory of moment maps in symplectic geometry \cite{kirwan_1984_convexity} and toric geometry \cite[Thm. 8.24]{michalek_2021_invitation} furthermore implies that the relation between these two parametrizations is a diffeomorphism. More precisely, the interior of $P(G)$ is diffeomorphic, under the inverse moment map, to the space of all $c$-weighted cut spaces of the graph, i.e., the $c$-scaled images of the coboundary operator $d^*_G:\R^V\rightarrow\R^E$. These vector subspaces form the real open part of a toric subvariety of the Grassmannian. This is a well-known perspective for matroid theorists since \cite{gelfand_1987_combinatorial}, but the interpretation of effective and relative resistances as images of the moment map seems to have gone unnoticed until recently in \cite{devriendt_2024_lives}.
\end{remark}
\begin{remark}[computation]
In practice, the relation between edge weights and relative resistances is efficiently computable. Given $c$, one can compute the relative resistances $r$ as the diagonal of the projection matrix onto the $c$-weighted cut space \cite{devriendt_2024_lives}. The hardest step in this procedure is a $\vert V\vert$-sized matrix inversion. Alternatively, there are highly optimized approximation algorithms \cite{spielman_2011_sparsification}. Given $r$, one can compute the corresponding edge weights iteratively: pick a random point $\hat{c}\in \R^E_{>0}$ and compute $\hat{r}(\hat{c})$. Then repeat the steps below until $\hat{r}\approx r$ as desired:
\begin{align*}
&\text{1. Let $\hat{c}_e \,\leftarrow\, \hat{c}_e\cdot (r_e/\hat{r}_e)^{1/(\vert V\vert-1)}$ for all $e\in E$}\\
&\text{2. Compute $\hat{r}(\hat{c})$}
\end{align*}
The theory of operator scaling guarantees convergence of the iterates in polynomial time \cite{garg_2020_operator}.
\end{remark}

We conclude this section by giving an explicit characterization of relative resistances based on the hyperplane description of $P(G)$. This greatly generalizes the bounds in \cite[Prop. 2.16-17]{devriendt_2022_thesis}.
\begin{corollary}\label{cor: r hyperplane inequalities}
Let $G$ be a graph. Then $r\in \R^E$ are relative resistances in $G$ if and only if
\begin{equation}\label{eq: hyperplane inequalities}
0\leq \sum_{e\in A}r_e \leq \vert V\vert - \textup{\#components of $(G-E\backslash A)$}\quad\quad\textup{for all $A\subseteq E$},
\end{equation}
with equality in the upper bound if and only if $A$ are the edges of an induced subgraph on some (possibly none or all) biconnected components of $G$, and equality in the lower bound if and only if $A$ is empty.
\end{corollary}
\begin{proof}
Expression \eqref{eq: hyperplane inequalities} are the hyperplane inequalities that cut out the spanning tree polytope $P(G)$ in $\R^E$; this follows for instance from the hyperplane description of matroid base polytopes in \cite[Prop. 2.3]{feichtner_2005_matroid}. By Foster's local theorem, the upper bound in \eqref{eq: hyperplane inequalities} is an equality when $A$ are the edges of an induced subgraph on biconnected components of $G$. This includes $A=E$, and $A=\emptyset$ where also the lower bound is attained and equal to zero.

Since the codimension of $P(G)$ is equal to the number of biconnected components of $G$ --- the corresponding matroid polytope has codimension equal to the number of connected components of the matroid \cite[Prop. 2.4]{feichtner_2005_matroid} --- no further independent equalities can hold on the relative interior of $P(G)$. Hence all other inequalities in \eqref{eq: hyperplane inequalities} must be strict, which completes the proof.
\end{proof}

Finally, we note that the results above extend to weights which can be zero or infinite. This closure is well-understood from the perspective of toric geometry, and it completes the space of possible relative resistances to the full spanning tree polytope $P(G)$. 

\section{Resistance nonnegative graphs}\label{sec: resistance nonnegative graphs}
We return to the main topic of this article. Recall the definition of resistance curvature.
\begin{definition}[resistance curvature] The \emph{resistance curvature} of a vertex $v\in V$ is
$$
p_v(c) = 1 - \frac{1}{2}\sum_{uv\in E}r_{uv}(c).
$$
\end{definition}
We make the definitions introduced in Question \ref{q: main question} explicit.
\begin{definition}[RN, RP, SRN]
A graph is \emph{resistance nonnegative} (RN) if there exist weights $0<c<\infty$ such that $p_v(c)\geq 0$ for all vertices $v$. A graph is \emph{resistance positive} (RP) if there exists weights such that $p_v(c)>0$ for all vertices $v$. A graph is \emph{strictly resistance nonnegative} (SRN) if it is RN but not RP.
\end{definition}
Applying the definition, we find that the diamond graph in Example \ref{ex: diamond graph} is RP, as certified by the relative resistances $(\tfrac{3}{5},\dots,\tfrac{3}{5})\in P(G_{\hdiamond})^{\circ}$. We further consider the examples in Figure \ref{fig: first examples}.
\begin{example}[RP graph]\label{ex: RP graph}
If $G$ is vertex-transitive, then for constant edge weights $c=1$ we find $p_v(1)=\vert V\vert^{-1}$ for all $v\in V$. This is shown for instance in \cite[Prop. 7]{devriendt_2022_curvature}. As a result, every vertex-transitive graph is RP. In particular, cycle graphs and the Petersen graph are RP.
\end{example}
\begin{example}[SRN graph]\label{ex: SRN graph}
For $n\in\N$, let $G=K_{n,n+1}$ be the near-balanced complete bipartite graph on vertex sets $V_n,V_{n+1}$, and pick constant edge weights $c=1$. Since $G$ is edge-transitive, $r_e(1)$ is constant and by Foster's theorem we find $p_v(1)=0$ for $v\in V_n$ and $p_{u}(1)=1/(n+1)$ for $u\in V_{n+1}$. Thus $K_{n,n+1}$ is RN.

Then, for any choice of edge weights $c$, we use Foster's theorem to find that
$$
\sum_{v\in V_{n}}p_v(c) = n - \tfrac{1}{2}\sum_{v\in V_{n}}\sum_{u\in V_{n+1}}r_{uv}(c) = n - \tfrac{1}{2}\sum_{e\in E}r_e(c) \overset{\text{(Foster)}}{=} n - \tfrac{1}{2}(2n+1 - 1) = 0.
$$
So in particular we cannot have $p_v(c)>0$ for all $v\in V_{n}$ and thus $K_{n,n+1}$ must be SRN. The same approach works in any bipartite graph, which leads to the following result:
\end{example}
\begin{proposition}\label{prop: RN RP and bipartite}
A bipartite graph $G$ is not RN if its bipartition sizes differ by more than one and $G$ is not RP if its bipartition sizes differ.
\end{proposition}
\begin{example}[not RN graph]\label{ex: not RN graph}
Let $G$ be a graph with cut vertex $x$, i.e., $G-x$ is disconnected. Following \cite[Prop. 2]{devriendt_2022_curvature}, we then have that $p_x(c)\leq 0$ with equality if and only if $x$ is an interior vertex in a path graph $G$. Thus the only non-biconnected RN graphs are paths, which are SRN.
\end{example}
\begin{proposition}
If $G$ is RN and not biconnected then $G$ is a path, and if $\vert V\vert\geq 3$ it is SRN.
\end{proposition}

We continue with a first characterization theorem that answers Question \ref{q: main question}.
\begin{theorem}\label{th: RN and expected degree}
A graph $G$ is RN (resp. RP) if and only if there exists a $\{$non-separable, log-linear, positive$\}$ distribution on $\mathcal{T}(G)$, such that the expected degree of every vertex in a random spanning tree under this distribution is at most $2$ (resp. strictly smaller than $2$). Any type of distribution is sufficient and all three types are necessary.
\end{theorem}
\begin{proof}
The expected degree of a vertex $v$ in a random spanning tree with distribution $\mu$ is
$$
\sum_{T}\mu(T)\sum_{uv\in T} 1 \,=\, \sum_{uv\in E}\Big(\sum_{T\ni uv}\mu(T)\Big).
$$
In other words, the expected degree is given in terms of the edge marginals of $\mu$. Let $\mu^*$ be a non-separable distribution for which the expected degree of each vertex is at most $2$. By Theorem \ref{thm: relative resistances and P(G)}, the edge marginals of $\mu^*$ are equal to relative resistances for some choice of edge weights $c^*$ and thus we find for every $v\in V$ that
$$
2\geq \sum_{uv\in E}\Big(\sum_{T\ni uv}\mu^*(T)\Big) = \sum_{uv\in E}r_{uv}(c^*) = 2-2p_v(c^*),
$$
which shows that $p(c^*)\geq 0$ and thus that $G$ is RN. Since non-separable distributions are the most general type, sufficiency of all three types follows. The converse statement and the strict version for RP graphs are proven in the same way.
\end{proof}

We now turn to a geometric version of this characterization. As the proof above illustrates, the expected degree constraints are affine inequalities on $\R^E$. We define the (open) halfspaces
$$
\mathcal{H}_v = \Big\{x \in\R^E \mid \sum_{uv\in E}x_{uv} \leq 2\Big\},\quad{\mathcal{H}_v}^{\circ} = \Big\{x\in\R^E\mid \sum_{uv\in E}x_{uv}<2 \Big\}\quad\quad \text{for~}v\in V.
$$
\begin{corollary}\label{cor: RN intersection halfspaces}
A graph $G$ is RN (resp. RP) if and only $P(G)^{\circ}$ and the halfspaces $\mathcal{H}_v$ (resp. open halfspaces $\mathcal{H}_v^{\circ}$) for all $v\in V$ have a common intersection point.
\end{corollary}
\begin{proof}
As in the proof of Theorem \ref{th: RN and expected degree}, a graph is RN if and only if it has relative resistances that satisfy $\sum_{uv\in E}r_{uv}\leq 2$ for all $v\in V$, equivalently, if $r\in\mathcal{H}_v$ for all $v$. Since $P(G)^{\circ}$ is the set of all possible relative resistances, this completes the proof; the same proof holds for RP.
\end{proof}
\begin{remark}[computation]\label{rem: computation}
Corollary \ref{cor: RN intersection halfspaces} shows that determining if a graph is RN can be done using linear programming. In practice, this requires some tolerance $\varepsilon$ to turn strict inequalities into non-strict ones. While spanning tree polytopes can have exponentially many (in $\vert E\vert$) defining inequalities, they admit a polynomial-sized extended formulation \cite{aprile_2021_smaller} using which these linear programs can be computed efficiently.
\end{remark}
\begin{question}
What is the computational complexity of the decision problems:
$$
\text{\emph{For a given graph $G$, decide if $G$ is RN/RP/SRN.}}
$$
\end{question}

\section{Trees \& double matchings}\label{sec: trees and double matchings}
We continue with the geometric characterization of RN and RP graphs in Corollary \ref{cor: RN intersection halfspaces}, and show that the intersection of $P(G)$ with the halfspaces $(\mathcal{H}_v)_{v\in V}$ has a combinatorial interpretation in terms of double matchings. Recall that a matching is a (possibly empty) subset of edges $M\subseteq E$ which covers every vertex at most once, and the matching polytope $M(G)$ is the convex hull of indicator vectors of matchings. Edmonds derived the hyperplane description of the matching polytope in \cite{edmonds_1965_matching}; scaled by a factor of two, this becomes:
\begin{equation}\label{eq: hyperplanes of 2M(G)}
2M(G) = \left\{x\in\R^E_{\geq 0} \,\,\bigg\vert\,\,
\begin{array}{l}
\sum_{uv\in E} x_{uv}\leq 2 \quad\quad\quad\quad\quad\text{~for all~} v\in V, \\
\sum_{uv\in E(U)}x_{uv}\leq 2\lfloor\tfrac{1}{2}\vert U\vert\rfloor \quad\text{~for all odd-sized~} U\subseteq V
\end{array} 
\right\}\subset \R^E,
\end{equation}
where $E(U)$ are the edges in the induced subgraph on $U$. Note that the $v$-based inequalities correspond to the expected degree condition. We show that the $U$-based inequalities are implied by the hyperplanes of $P(G)$, which leads to the following characterization of RN graphs:
\begin{theorem}\label{thm: intersecting trees and matchings}
A graph $G$ is RN if and only if $P(G)^{\circ}\cap 2M(G)$ is non-empty.
\end{theorem}
\begin{proof}
If a point $r\in P(G)^{\circ}$ lies in $2M(G)$, then by the $v$-based inequalities of $2M(G)$ we know that r lies in the intersection of $(\mathcal{H}_v)_{v\in V}$ and thus that $G$ is RN by Corollary \ref{cor: RN intersection halfspaces}.

If $G$ is RN then there is a point $r\in P(G)^{\circ}\subset \R^E_{\geq 0}$ which satisfies the $v$-based inequalities of $2M(G)$. It remains to show that $r$ also satisfies the $U$-based inequalities. Let $U\subseteq V$ have odd size and $G[U]$ be the induced subgraph on $U$ with edges $E(U)$. We have
\begin{align*}
\sum_{e\in E(U)} r_e &\leq \vert V\vert - \textup{\#components of $\big(G-E\setminus E(U)\big)$} \quad\text{(Corollary \ref{cor: r hyperplane inequalities})}
\\[-0.6\baselineskip]
&= \vert U\vert - \text{\#components of $\big(G[U]\big)$} \,\leq\, \vert U\vert - 1 \,=\, 2\lfloor\tfrac{1}{2}\vert U\vert\rfloor.
\end{align*}
This confirms that $r\in 2M(G)$ and completes the proof.
\end{proof}

\begin{remark}\label{rem: space of metrics}
Similar to how $P(G)^{\circ}$ parametrizes all possible edge weights, the set $P(G)^{\circ}\cap2M(G)$ parametrizes all edge weights for which $G$ has nonnegative resistance curvature. Geometrically, thinking of edge weights as inverse edge lengths, this is the ``space of metrics with nonnegative curvature on $G$". The fact that the closure of this set is a polytope is in stark contrast with the complexity of the ``space of metrics with positive/nonnegative curvature on a manifold" as studied in differential geometry; see for instance \cite[Sec. 2]{rosenberg_2006_manifolds}.
\end{remark}

A matching in a graph can have size at most $\lfloor\tfrac{1}{2}\vert V\vert\rfloor$. Such a matching is called \emph{perfect} if the graph is even, and \emph{near-perfect} if the graph is odd. Perfect matchings in graphs and their statistics are very actively studied; they are also called \emph{dimers} \cite{kenyon_2009_lectures}. It follows from $1$-tough that RP graphs have a (near-)perfect matching, but we can strengthen this result to all RN graphs.
\begin{corollary}\label{cor: RN has perfect matching}
Every resistance nonnegative graph has a (near-)perfect matching.
\end{corollary}
\begin{proof}
By Foster's theorem, the spanning tree polytope lies in the hyperplane of points in $\R^E$ whose entries sum to $\vert V\vert -1$. If a graph $G$ is RN, then by Theorem \ref{thm: intersecting trees and matchings} there must be a point in $2M(G)$ that lies in this hyperplane. Since this point is a convex combination of indicator vectors of double matchings, $G$ must have a matching of size at least $\lceil\tfrac{1}{2}(\vert V\vert -1)\rceil = \lfloor\tfrac{1}{2}\vert V\vert\rfloor$.
\end{proof}

We note that Corollary \ref{cor: RN has perfect matching} is useful as a necessary condition to check if a class of graphs could be RN. It can be used for instance to show that imbalanced bipartite graphs are not RN, as in Proposition \ref{prop: RN RP and bipartite}. The converse to Corollary \ref{cor: RN has perfect matching} does not hold. For instance, a 3-cycle glued to a 4-cycle at a single vertex has a near-perfect matching, but it is not RN (see Example \ref{ex: not RN graph}).
\\~\\
Motivated by Theorem \ref{thm: intersecting trees and matchings} and Remark \ref{rem: space of metrics}, we take a better look at the following polytope:
\begin{definition}
The \emph{tree double matching} (TDM) polytope of a graph is $\Theta(G) = P(G)\cap2M(G)$.
\end{definition}
\begin{example}[even cycle graphs]\label{ex: Theta(G) for cycle}
For $n=2k\geq 3$ even, the cycle graph $C_n$ has $n$ spanning trees, $2$ perfect matchings of size $k$ and $n$ matchings of size $k-1$. Each of the spanning trees can be composed by combining a $(k-1)$-sized matching and one of the two $k$-sized matching. The TDM-polytope is the whole spanning tree polytope $\Theta(C_n)=P(C_n)$. 
\end{example}
Integer points of polytopes are important in many areas of combinatorics and its applications; a polytope is called a \emph{lattice polytope} if its vertices are integer points. For the TDM-polytope, we find that the integer points $\Theta(G)\cap \Z^E$ have a simple characterization; recall that a \emph{Hamiltonian path} in a graph is a spanning tree which is also a path.

\begin{proposition}\label{prop: integer points are hamiltonian paths}
The integer points of $\Theta(G)$ are indicator vectors of Hamiltonian paths.
\end{proposition}
\begin{proof}
Since the vertices of $P(G)$ lie in $\{0,1\}^E$ the only integer points in $P(G)$ are its vertices, which are indicator vectors of spanning trees. A vertex $e_T$ of $P(G)$ lies in $\Theta(G)$ if and only if it can be written as a convex combination of double matchings in $G$. By Foster's theorem as in the proof of Corollary \ref{cor: RN has perfect matching} and integrality of the point $e_T$, this point must be a combination of (i) two matchings with $\tfrac{1}{2}(\vert V\vert-1)$ edges if $G$ is odd, or (ii) one matching with $\tfrac{1}{2}\vert V\vert$ edges and another one with $\tfrac{1}{2}\vert V\vert-1$ edges if $G$ is even. Combining these matchings gives a subgraph with degree at most $2$, which is connected if and only if it is a path. This completes the proof.
\end{proof}

We can use this result to give a combinatorial condition for the RN property; a set of Hamiltonian paths $\mathcal{F}$ is called \emph{independent} if the indicator vectors $(e_T)_{T\in\mathcal{F}}$ are affinely independent.
\begin{corollary}
If $G$ contains $(\vert E\vert-$\#biconnected components of $G)$ independent Hamiltonian paths, then $G$ is RN.
\end{corollary}
\begin{proof}
The result is true if $G$ is a path graph so we may assume that $G$ is biconnected. If $G$ contains $\vert E\vert -1$ independent Hamiltonian paths, then $\conv(e_T\mid T\text{~is Hamiltonian path})\subseteq \Theta(G)$ has dimension $\vert E\vert-1$. Since this is equal to the dimension of $P(G)$, this convex hull and thus also $\Theta(G)$ must intersect $P(G)^{\circ}$. By Theorem \ref{thm: intersecting trees and matchings} this completes the proof.
\end{proof}

Proposition \ref{prop: integer points are hamiltonian paths} and its proof illustrate how the integer points of the TDM polytope correspond combinatorially to trees that can be constructed as the union of two matchings. A similar combinatorial interpretation holds for rational points of $\Theta(G)$ or, equivalently, integer points of its dilations. We propose the following Ehrhart-type question for TDM polytopes:
\begin{question}
Count the number of integer points $n_k=\#( k\Theta(G)\cap\Z^E)$ for $k\in \N$.
\end{question}

Let us return to Example \ref{ex: Theta(G) for cycle}, where we mentioned that $\Theta(G)$ coincides with $P(G)$ for cycle graphs. It turns out that this is the only class of graphs where this occurs.
\begin{corollary}\label{cor: characterizations of cycle graph}
The following are equivalent for a biconnected graph $G$:
\begin{enumerate}
\item $\Theta(G)=P(G)$;
\item $p(c)>0$ for all weights $0<c<\infty$;
\item $G$ is a cycle or a single edge $G=K_2$.
\end{enumerate}
\end{corollary}
\begin{proof}
Note that $\Theta(G)=P(G)$ is equivalent to $P(G)^{\circ}\cap 2M(G)=P(G)^{\circ}$, which says that every point $r\in P(G)^{\circ}$ lies in the halfspaces $(\mathcal{H}_v)_{v\in V}$. This establishes $1.\Leftrightarrow 2.$ The implication $3.\Rightarrow 1.$ follows from Proposition \ref{prop: integer points are hamiltonian paths}. Conversely, if $\Theta(G)=P(G)$ then all vertices of $P(G)$ must be Hamiltonian paths, so all spanning trees of $G$ are paths. For a biconnected graph, this occurs if and only if $G$ is a cycle or a single edge, which proves $1.\Rightarrow 3.$
\end{proof}

We note that paths with more than $3$ vertices satisfy similar properties: $\Theta(G)=P(G)=e_G$ and $p(c)\geq 0$ for all weights. We propose to study the vertices of general TDM-polytopes:
\begin{question}
What are the vertices of $\Theta(G)$? What is the smallest $k$ such that the $k$-dilated TDM polytope is a lattice polytope, i.e., determine $k_G = \min\{k\in \N\mid k\Theta(G)\text{~is lattice polytope}\}$.
\end{question}

More generally, the points in $\Theta(G)$ are edge marginals of distributions $\mu$ on spanning trees in $G$ and distributions $\lambda$ on matchings in $G$, at the same time. In both cases, these distributions can be taken log-linear $\mu_c$ and $\lambda_{\tilde{c}}$ induced by some weights $c$ and $\tilde{c}$ on the edges. 
\begin{question}
What is the relation between the edge weights $c$ and $\tilde{c}$, for which the distributions $\mu_c$ on spanning trees and $\lambda_{\tilde{c}}$ on matchings have the same edge marginals $x\in\Theta(G)$?
\end{question}
To conclude the section, we revisit Example \ref{ex: diamond graph} and illustrate the new definitions and results.
\begin{example}[diamond graph, continued]\label{ex: diamond graph matchings}
Figure \ref{fig: matching polytope} shows the $8$ matchings of the diamond graph arranged according to the $1$-skeleton of the matching polytope $M(G_{\hdiamond})\subset \R^5$. Since the diamond graph is RP, the intersection $P(G_{\hdiamond})^{\circ}\cap2M(G_{\hdiamond})$ is non-empty. The corresponding TDM polytope $\Theta(G_{\hdiamond})$ is a lattice polytope whose $6$ vertices are indicator vectors of the Hamiltonian paths of $G_{\hdiamond}$. The number of rational points in the TDM polytope can be found from its Ehrhart polynomial; we compute this polynomial using \texttt{Polymake} \cite{gawrilow_2000_polymake} as
$$
\#(k\Theta(G_{\hdiamond})\cap\Z^5) = 1+\frac{9}{4}k + \frac{15}{8}k^2+\frac{3}{4}k^3+\frac{1}{8}k^4 \text{~for $k\in\N$.}
$$
\begin{figure}[h!]
    \centering
    \includegraphics[width=0.36\linewidth]{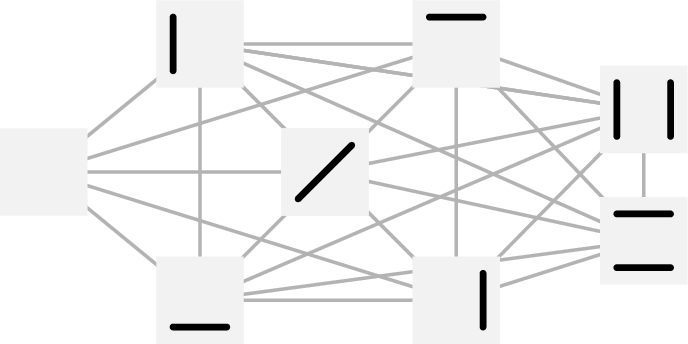}
    \caption{The $1$-skeleton of the matching polytope $M(G_{\hdiamond})$.}
    \label{fig: matching polytope}
\end{figure}
\end{example}

\section{Algebraic characterization}\label{sec: algebraic characterization}
Until here, we have used the relation between effective resistances and spanning trees to characterize RN and RP graphs. However, the effective resistance has a rich algebraic theory independent of the combinatorics of spanning trees and matroids; in this section, we use this theory to derive one further characterization. More precisely, we start from an alternative parametrization for the space of all possible effective resistances in a graph --- a space $\mathcal{K}(G)$ with a similar role to $P(G)^\circ$ --- and derive an analogue of Corollary \ref{cor: RN intersection halfspaces} to characterize RN and RP graphs. Interestingly, the condition $p\geq 0$ again translates to affine inequalities on $\mathcal{K}(G)$. 

We give a minimal treatment of Laplacian and resistance matrices and refer to \cite[Ch. 3]{devriendt_2022_thesis} for more details. Let $G$ be a connected graph with $n=\vert V\vert$ vertices and edge weights $0<c<\infty$. The Laplacian matrix of $G$ is the $n\times n$ symmetric matrix $L$ with entries
$$
L_{uv} = 0 \text{~for $uv\not\in E$},  \quad\quad  
L_{uv}= -c_{uv}\text{~for $uv\in E$},  \quad\quad
L_{vv}= -\sum_{u\in V} L_{uv}\text{~for $v\in V$}.
$$
This matrix is real, symmetric, positive semidefinite and has $\ker(L)=\spn(\mathbf{1})$ , where $\mathbf{1}\in\R^n$ is the vector of all ones. Together with off-diagonal nonpositivity, these four properties characterize Laplacian matrices \cite[Prop. 2.5]{devriendt_2022_thesis}. Let $L^{\dagger} = (L+\tfrac{1}{n}\mathbf{1}\mathbf{1}^T)^{-1}-\tfrac{1}{n}\mathbf{11}^T$ be the Moore--Penrose pseudoinverse of $L$, then the effective resistance between two vertices is defined as
\begin{equation}\label{eq: definition effective resistances pseudoinverse}
\omega_{uv} = L^{\dagger}_{uu} + L^{\dagger}_{vv} - 2L^{\dagger}_{uv}, \quad\quad\text{~for $u,v\in V$}.
\end{equation}
For adjacent vertices, this agrees with definition \eqref{eq: definition effective resistance} in the introduction and we recall that $r_{uv}=\omega_{uv}c_{uv}$ for $uv\in E$. The \emph{resistance matrix} of $G$ is the $n\times n$ matrix $\Omega=(\omega_{uv})_{u,v\in V}$. It is a real, symmetric, invertible matrix with zero diagonal \cite[Prop. 3.2 \& 3.5]{devriendt_2022_thesis}. We will further work with the inverse resistance matrix $K=\Omega^{-1}$ and consider a special class of edge weights:
\begin{definition}
The edge weights $0<c<\infty$ are called \emph{normalized} if $\sum_{u,v\in V}K_{uv}=1$.
\end{definition}
We call a graph with normalized weights a \emph{normalized graph}. Assuming normalized edge weights is not restrictive, since any choice of edge weights can be normalized by the common rescaling $c_e\mapsto c_e / (\sum_{u,v\in V}K_{uv})$ for all $e\in E$. The resistance curvature of a normalized graph has a very simple expression in terms of $K$.
\begin{proposition}[{\cite[Cor. 3.10]{devriendt_2022_thesis}}]\label{prop: p in terms of K}
Let $G$ be a normalized graph, then $p_v(c) = \sum_{u\in V}K_{uv}$.
\end{proposition}

We now rephrase \cite[Thm. 3.9]{devriendt_2022_thesis} and characterize all possible inverse resistance matrices associated to a given graph as a semialgebraic set inside $\GLnsym(\R)$, the set of real symmetric invertible $n\times n$ matrices. We think of this as an alternative parametrization to $P(G)^{\circ}$.
\begin{theorem}[\cite{devriendt_2022_thesis}]
The set of inverse resistance matrices of $G$ with normalized weights is
$$
\mathcal{K}(G) = \left\{ K\in\GLnsym(\R) \,\Biggl\vert\,
\begin{array}{l}
\sum_{u,v\in V}K_{uv}=1\\
K_{uv} = \big(\sum_{w\in V}K_{wv}\big)\big(\sum_{w\in V}K_{uw}\big)\quad \text{~for all $uv\not\in E$}\\
K_{uv} > \big(\sum_{w\in V}K_{wv}\big)\big(\sum_{w\in V}K_{uw}\big)\quad \text{~for all $uv\in E$}
\end{array}
\right\}.
$$
\end{theorem}
\begin{proof}
We summarize the proof of \cite[Thm. 3.9]{devriendt_2022_thesis}. Given $K\in\mathcal{K}(G)$, we construct the matrix
\begin{equation}\label{eq: L and K}
L = -2K + 2(K\mathbf{1})(K\mathbf{1})^T.
\end{equation}
It can be checked that this matrix is real, symmetric, positive semidefinite with $\ker(L)=\operatorname{span}(\mathbf{1})$ and that it has nonpositive off-diagonal entries: $L_{uv}<0$ if $uv\in E$ and zeros otherwise. Thus, $L$ is a Laplacian matrix of $G$ with weights $c^*_{uv}=-L_{uv}$, and $K$ is the inverse resistance matrix of this graph by \cite[Cor. 3.8]{devriendt_2022_thesis}. Here, $K\in\mathcal{K}(G)$ confirms that the weights $c^*$ are normalized.

Conversely, it can be checked that every inverse resistance matrix of $G$ lies in the set $\mathcal{K}(G)$; this follows for instance by relation \eqref{eq: L and K} between the Laplacian and inverse resistance matrix of a normalized graph, and then using the properties of Laplacian matrices. 
\end{proof}

We now formulate an analogue of Corollary \ref{cor: RN intersection halfspaces} based on the parametrization $\mathcal{K}(G)$. As before, we define (open) halfspaces to keep track of signs of the resistance curvature
$$
\widetilde{\mathcal{H}}_v = \Big\{K\in\GLnsym(\R) \mid \sum_{u\in V}K_{uv} \geq 0\Big\}, \quad \widetilde{\mathcal{H}}_v^{\circ} = \Big\{K\in\GLnsym(\R) \mid \sum_{u\in V}K_{uv}>0\Big\}\quad\text{~for $v\in V$}.
$$
\begin{corollary}\label{cor: RN from inverse resistance matrices}
A graph $G$ is RN (resp. RP) if and only if $\mathcal{K}(G)$ and the halfspaces $\widetilde{\mathcal{H}}_v$ (resp. open halfspaces $\widetilde{\mathcal{H}}_v^{\circ}$) for all $v\in V$ have a common intersection point.
\end{corollary}
\begin{proof}
We may assume that all weights are normalized, without loss of generality. The corollary then immediately follows from the expression of $p$ in terms of $K$ in Proposition \ref{prop: p in terms of K}.
\end{proof}
\begin{example}
The Laplacian, resistance and inverse resistance matrix of $G_{\hdiamond}$ with $c=1$  are
$$
L = \begin{pmatrix}
2 & -1 & 0 & -1\\
-1 & 3 & -1 & -1\\
0 & -1 & 2 & -1\\
-1 & -1 & -1 & 3
\end{pmatrix}\quad\quad\Omega = \frac{1}{8}\begin{pmatrix}
0 & 5 & 8 & 5\\
5 & 0 & 5 & 4\\
8 & 5 & 0 & 5\\
5 & 4 & 5& 0
\end{pmatrix}\quad\quad K = \frac{1}{1.36}\begin{pmatrix}
-1 & 0.8 & 0.36 & 0.8\\
0.8 & -2 & 0.8 & 0.72\\
0.36 & 0.8 & -1 & 0.8\\
0.8 & 0.72 & 0.8 & -1
\end{pmatrix}
$$
\end{example}

\section{Graph operations}\label{sec: graph operations}
We consider three graph operations and their effect on resistance nonnegativity and positivity. Figure \ref{fig: graph operations} at the end of the section illustrates the three operations.
\subsection{Subgraphs}
Since adding edges to a graph increases its set of spanning trees, we obtain the following:
\begin{lemma}\label{lem: adding edges}
Let $H\subseteq G$ with $V(H)=V(G)$ and $H$ connected. If $H$ is RP then $G$ is RP.
\end{lemma}
\begin{proof}
Let $H$ be a connected RP subgraph of $G$ on the same vertex set. The spanning tree sets of these two graphs satisfy $\mathcal{T}(H)\subseteq\mathcal{T}(G)$. Since $H$ is RP, there exists a positive distribution $\mu$ on $\mathcal{T}(H)$ whose edge marginals $r\in P(H)^{\circ}$ satisfy $\sum_{uv\in E(H)}r_{uv}< 2$ for all $v\in V$. We then define a distribution $\mu^*$ on $\mathcal{T}(G)$ as
$$
\mu^*:T\mapsto\begin{cases}
(1-\varepsilon)\cdot\mu(T)\quad\quad\quad\quad\,\,\,\,\text{~if $T\in\mathcal{T}(H)$}\\
\varepsilon\cdot \big(\vert\mathcal{T}(G)\setminus\mathcal{T}(H)\vert\big)^{-1}\quad\text{~if $T\not\in\mathcal{T}(H)$},
\end{cases}\quad\quad\text{~for $T\in\mathcal{T}(G)$.}
$$
This is a positive distribution on $\mathcal{T}(G)$ with edge marginals $r^*\in P(G)^{\circ}$. For sufficiently small $\varepsilon>0$ these marginals satisfy $\sum_{uv\in E(G)}r^*_{uv}< 2$ for all $v$, by linearity of $r,r^*$ in terms of $\mu,\mu^*$. 
\end{proof}

A graph $G$ is \emph{Hamiltonian} if it contains a $\vert V(G)\vert$-length cycle graph as a subgraph.
\begin{theorem}\label{thm: Hamiltonian implies RP}
Hamiltonian graphs are resistance positive, but the converse is not always true.
\end{theorem}
\begin{proof}
If $G$ is Hamiltonian, then it contains a cycle graph $H\subseteq G$ as a subgraph on the same vertex set. Since $H$ is connected and RP, $G$ is also RP by Lemma \ref{lem: adding edges}. The Petersen graph is RP (since it is vertex-transitive) but not Hamiltonian, so RP does not imply Hamiltonicity.
\end{proof}

The study of Hamiltonicity is closely related to graph toughness. Recall that a graph is $t$-tough if for every set of vertices $U\subset V$ whose removal disconnects the graph, the graph $G-U$ has at most $\vert U\vert/t$ components; cliques are $\infty$-tough. Chv\'{a}tal conjectured the following:
\begin{conjecture}[{\cite[Conj. 2.3]{chvatal_1973_tough}}]\label{conj: chvatal conjecture} There exists $t_0$ such that every $t_0$-tough graph is Hamiltonian.
\end{conjecture}
We refer to the survey \cite{bauer_2006_toughness} for history and progress on this conjecture. Combining Theorem \ref{thm: RP implies 1-tough} and Theorem \ref{thm: Hamiltonian implies RP}, we find that RP graphs lie between two graph classes: \emph{Hamiltonian $\subset$ RP $\subseteq$ 1-tough}. Following this relation, we propose a relaxation of Chv\'{a}tal's conjecture:
\begin{conjecture}\label{conj: toughness conjecture}
There exists $t_{RP}$ such that every $t_{RP}$-tough graph is resistance positive.
\end{conjecture}
\begin{proposition}
Conjecture \ref{conj: chvatal conjecture} implies Conjecture \ref{conj: toughness conjecture}.
\end{proposition}
\begin{proof}
If Conjecture \ref{conj: chvatal conjecture} holds for $t_0$, then $t_0$-tough implies Hamiltonicity which implies RP.
\end{proof}
The latter conjecture was also posed for $t_{RP}=1$ by Fiedler in \cite[p59]{fiedler_2011_matrices}. We ask the following:
\begin{question}
Does $1$-tough imply $RP$? Is every non-Hamiltonian $1$-tough graph RP?
\end{question}
~\\
As a further application we consider grid graphs. The cartesian product $G\times H$ of two graphs has vertex set $V(G)\times V(H)$ and edge set $\{(u,x)(v,y)\mid uv\in E(G)\text{~or~}xy\in E(H)\}$, where ``or" is exclusive. A \emph{grid graph} is the cartesian product of two path graphs $P_n\times P_m$. In \cite{dawkins_2024_resistance} it was shown that for constant edge weights $c=1$ and $m,n\geq 3$, all grid graphs have negative resistance curvature on interior vertices and nonnegative resistance curvature on boundary vertices and if $m>3$ or $n>3$, then the boundary vertices have positive curvature. Here we find that some of these grid graphs are RP. This illustrates the relevance of considering edge weights.
\begin{proposition}
For $m\cdot n$ even and $m,n>1$, the grid graph $P_n\times P_m$ is RP.
\end{proposition}
\begin{proof}
For these parameters, $P_n\times P_m$ is Hamiltonian and thus RP.
\end{proof}

We conjecture that all grid graphs are resistance nonnegative, and ask about generalizations.
\begin{conjecture}[Grid graphs are RN]
$P_n\times P_m$ is RN for all $m,n\in \N$.
\end{conjecture}
\begin{question}
Are all graphs $P_n\times \dots\times P_m$ RN? Which biconnected planar graphs are RN?
\end{question}

\subsection{Kron reduction}
Kron reduction is an important tool in electrical circuit analysis, which allows to reduce a large circuit to a smaller one while maintaining its electrical properties \cite{dorfler_2013_kron}. Combinatorially, this operation is very simple; we write $U^c=V\setminus U$ and $\partial U=\{v\in U^c\mid uv\in E \text{~for some $u\in U$}\}$.
\begin{definition}[Kron reduction]\label{def: definition Kron reduction}
Let $G$ be a connected graph with vertex subset $U\subset V$. The \emph{Kron reduction} of $U$ in $G$ is the graph obtained by removing $U$ and adding all possible edges between $\partial U$; we write $\kron_U(G)$.
\end{definition}
\begin{definition}[matched weights]
Let $\widetilde{G}$ be the Kron reduction of $x\in V$ in $G$ with edge weights $c$. The \emph{matched weights} $\tilde{c}$ in $\widetilde{G}$ are defined as
$$
\tilde{c}_{uv} = c_{uv}+\frac{c_{ux}\cdot c_{xv}}{\sum_{xy\in E(G)}c_{xy}} \text{~if $u,v\in\partial x$,}\quad\quad\quad \tilde{c}_{uv}=c_{uv}\text{~otherwise,}\quad\quad\quad\text{~for $uv\in E(\widetilde{G}).$}
$$
\end{definition}
We recall some relevant properties of Kron reductions; see \cite[Ch. 3]{devriendt_2022_thesis},\cite{dorfler_2013_kron} for more details.
\begin{lemma}[quotient property, \cite{dorfler_2013_kron}]
Let $U=\{u_1,\dots,u_k\}\subset V(G)$ in arbitrary order. Then 
\begin{equation}\label{eq: composition}
\kron_U(G)=\kron_{u_k}\circ\dots\circ\kron_{u_1}(G).
\end{equation}
\end{lemma}
To understand Kron reduction it often suffices to understand Kron reduction of single vertices. For a general vertex subset $U\subset V$, if the weights are matched at every step in the decomposition of $\widetilde{G}=\kron_U(G)$ along single vertices as in \eqref{eq: composition}, then we say that the final weights $\tilde{c}$ on $E(\widetilde{G})$ are matched; these weights can also be obtained directly from the Schur complement of the associated Laplacian matrix. The following is relevant for RN graphs:
\begin{lemma}[{\cite[Prop. 3.25]{devriendt_2022_thesis}}]\label{lem: resistance curvature and kron}
Let $\widetilde{G}=\kron_x(G)$ with matched edge weights $\tilde{c}$. Then
$$
\tilde{p}_v(\tilde{c}) = p_v(c) + \frac{c_{xv}\cdot p_x(c)}{\sum_{xy\in E(G)}c_{xy}}\text{~if $v\in\partial x$,}\quad\quad\quad \tilde{p}_v(\tilde{c}) = p_v(c)\text{~otherwise,}\quad\quad\text{~for $v\in V(\widetilde{G})$}.
$$
\end{lemma}
In the above lemma, we note that if $p_v(c),p_x(c)\geq 0$ then also $\tilde{p}_v(\tilde{c})\geq 0$, and similarly if all inequalities are strict. This leads to the following closure property:
\begin{proposition}
If $G$ is RN (resp. RP), then any Kron reduction of $G$ is RN (resp. RP).
\end{proposition}
\begin{proof}
Let $G$ be an RN graph with edge weights $c$ such that $p(c)\geq 0$ and take any $U\subset V$. Then by taking Kron reductions of single vertices as in \eqref{eq: composition} with matched edge weights, and using Lemma \ref{lem: resistance curvature and kron}, we find that $\kron_U(G)$ is again RN. The same derivation holds for RP graphs.
\end{proof}

\subsection{Circle inversion}
Different from the other two operations, the next operation is defined with respect to a choice of edge weights; it originates from Fiedler's matrix theory of simplices, where it corresponds to spherical inversion of a simplex around one of its vertices \cite[Rem. 3.4.5]{fiedler_2011_matrices}.
\begin{definition}[circle inversion, \cite{fiedler_2011_matrices}]\label{def: definition circle inversion}
Let $G$ be an RN graph with vertex $x\in V$ and weights $c$ such that $p(c)\geq 0$. The \emph{circle inversion over $x$} is the graph obtained by removing $x$ and adding a new vertex $\hat{x}$ with edges $\{ \hat{x}y\mid p_y(c)>0\text{~in $G$}\}$; we write $\inv_x(G,c)$.
\end{definition}
Similar to Kron reduction, circle inversion comes with a distinguished choice of edge weights. We make use of the effective resistance $\omega$ defined as in \eqref{eq: definition effective resistances pseudoinverse} between any pair of vertices.
\begin{definition}
Let $\widetilde{G}$ be the circle inversion over $x$ with respect to weights $c$. The \emph{matched weights} $\tilde{c}$ in $\widetilde{G}$ are defined as
$$
\tilde{c}_{\hat{x}y} = 2\cdot p_{y}(c)\cdot \omega_{xy}\text{~if $y\in\partial\hat{x}$,}\quad\quad
\tilde{c}_{uv}=c_{uv}\cdot \omega_{ux}\cdot \omega_{xv}\text{~if $u,v\not\in\partial \hat{x}$}.
$$
\end{definition}

These weights are only nonnegative if $p(c)\geq 0$. When fixing a choice of edge weights on a graph and keeping these weights matched under circle inversions, the map $\inv_{\bullet}:(G,c)\mapsto(\widetilde{G},\tilde{c})$ commutes and has an inverse, i.e., $\inv_x\circ \inv_y = \inv_y \circ \inv_x$ up to relabeling vertices, and $\inv_{\hat{x}}\circ\inv_{x}=\operatorname{Id}$. The relevance of circle inversion for RN graphs is due to the following result:
\begin{lemma}[{\cite[Thm. 3.4.4]{fiedler_2011_matrices}}]\label{lem: resistance curvature inversion}
Let $\widetilde{G}=\inv_x(G,c)$ with matched edge weights $\tilde{c}$. Then
$$
\tilde{p}_v(\tilde{c}) = \tfrac{1}{2}\cdot c_{xv}\cdot\omega_{xv} \text{~if $v\in\partial x$,}\quad\quad 
\tilde{p}_{v}(\tilde{c}) = p_x(c)\text{~if $v=\hat{x}$,} \quad\quad
\tilde{p}_v(\tilde{c}) = 0\text{~otherwise,} \quad\text{~for $v\in V(\widetilde{G})$}.
$$
\end{lemma}
This directly leads to another closure property for RN and RP graphs.
\begin{proposition}
If $G$ is RN, then any circle inversion of $G$ is RN. Furthermore, if $\inv_x(G,c)$ is over a vertex $x$ with $p_x(c)>0$ and which is connected to all other vertices, then it is RP.
\end{proposition}
\begin{figure}[h!]
    \centering
    \includegraphics[width=0.8\linewidth]{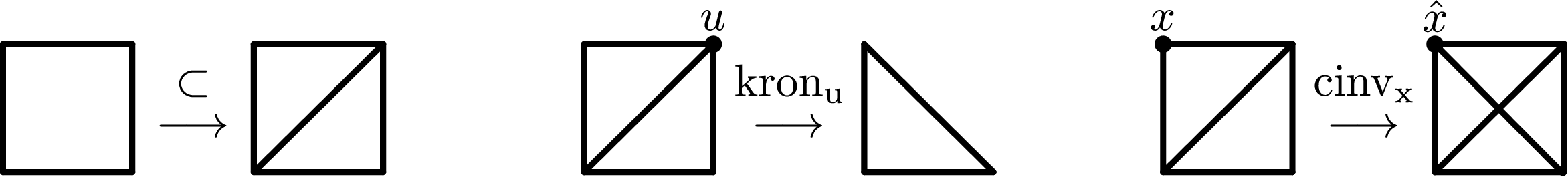}
    \caption{Example of a subgraph, Kron reduction and circle inversion (with $c=1$).}
    \label{fig: graph operations}
\end{figure}

\section{Nonnegativity from submodularity}\label{sec: nonnegativity and submodularity}
We conclude the article with a speculative third approach to characterizing resistance nonnegativity. We again start from a space $\mathcal{S}(G)$ that parametrizes all possible effective resistances in a graph and conjecture that the condition $p\geq 0$ translates to linear inequalities on $\mathcal{S}(G)$; more precisely, we conjecture that these inequalities are those of the submodular cone.

For $U\subseteq V$ nonempty, we write $\Omega[U]$ for the submatrix of the resistance matrix $\Omega$ with rows and columns in $U$. A variant of the following set function was introduced in \cite[Ch. 3.4]{devriendt_2022_thesis}.
\begin{definition}[resistance capacity]\label{def: resistance capacity}
For a normalized graph $G$, the \emph{resistance capacity} is a set function $\tau_{\bullet}(c):2^V\rightarrow\R$ defined on the subsets of $V$ as: $\tau_{\emptyset}(c)=0$, and $\tau_{v}(c)=\tfrac{1}{2}$ for $v\in V$, and
$$
\tau_{U}(c) = \frac{1}{2}+\frac{1}{2}\Big(\sum_{u,v\in U} (\Omega[U]^{-1})_{uv}\Big)^{-1} \quad\quad\text{~for $U\subseteq V$ with $\vert U\vert\geq 2$.}
$$
\end{definition}
\begin{remark}
The resistance capacity $\tau_{V}$ is a discrete analogue of a well-studied invariant in potential theory on metric graphs, the so-called tau constant of a metric graph \cite[Thm. 1.27]{cinkir_2007_thesis}.
\end{remark}

A set function $f$ is called \emph{monotone} if it satisfies $f(A)\leq f(B)$ for all $A\subseteq B$.
\begin{proposition}
The resistance capacity is monotone and lies in $[0,1]$.
\end{proposition}
\begin{proof}
Monotonicity was proven in \cite[Prop 3.38]{devriendt_2022_thesis}. For $U=\emptyset$ and $U=V$ we directly compute $\tau_\emptyset=0$ and $\tau_{V}=1$, and by monotonicity this implies that $\tau_U\in [0,1]$ for all $U\subseteq V$.
\end{proof}

Next, we observe that the resistance capacity parametrizes all data of a weighted graph:
\begin{proposition}
The graph $G$ and edge weights $c$ can be recovered from $(\tau_U(c))_{U\subseteq V}$.
\end{proposition}
\begin{proof}
For $U=uv$, the value of $\tau$ was computed in \cite[p53]{devriendt_2022_thesis} to be $\tau_{uv}(c)=\tfrac{1}{2}+\tfrac{1}{4}\omega_{uv}$. We can thus recover all effective resistances in the graph as $\omega_{uv}=4\tau_{uv}(c)-2$. From this data, the graph and edge weights can be found by constructing the Laplacian matrix via $K=\Omega^{-1}$ as in \eqref{eq: L and K}.
\end{proof}

We thus have a third parametrization in addition to $P(G)^{\circ}$ and $\mathcal{K}(G)$: the set of resistance capacities $(\tau_U(c))_{U\subseteq V}$ for all normalized edge weights $c$ on a given graph $G$. We denote this set by $\mathcal{S}(G)\subseteq [0,1]^{2^V}$. We now relate certain inequalities on this space to resistance nonnegativity. A set function $f:2^V\rightarrow \R$ is called \emph{submodular} if it satisfies
\begin{equation}\label{eq: definition submodularity}
f(A) + f(B) \geq f(A\cup B) + f(A\cap B) \quad\quad\quad\text{~for all $A,B\subseteq V$.}
\end{equation}
The set of all submodular set functions on a given ground set $V$ is called the submodular cone; it is a polyhedral cone in $\R^{2^V}$ cut out by the submodularity inequalities \eqref{eq: definition submodularity}. We rephrase \cite[Thm. 6.33]{devriendt_2022_thesis}, which implies submodularity for the resistance curvature when $p(c)\geq 0$:
\begin{theorem}
Let $G$ and $c$ be normalized with $p(c)\geq 0$. Then $\tau_{\bullet}(c)$ is submodular.
\end{theorem}
\begin{proof}
Let $\sigma^2_{U} = \tau_{U} - \tfrac{1}{2}\delta_{U\neq\emptyset}$ be the resistance capacity shifted by $-\tfrac{1}{2}$ for all sets except for the empty set $U=\emptyset$. It was shown in \cite[Thm. 6.33]{devriendt_2022_thesis} that $\sigma^2$ satisfies the submodularity inequality \eqref{eq: definition submodularity} whenever the two sets $A,B$ intersect. The main obstruction in proving full submodularity of $\sigma^2$ is that inequality \eqref{eq: definition submodularity} fails for distinct singletons $A=\{a\},B=\{b\}$. Indeed, we have
$$
\sigma^2_a(c)+\sigma^2_b(c) = 0 < \sigma^2_{ab}(c).
$$
This obstruction is no longer present for $\tau$ due to the $-\tfrac{1}{2}$ shift and normalization; we find
$$
\tau_a(c) + \tau_b(c) = 1 = \tau_E(c) \geq \tau_{ab}(c)\quad\quad\text{~for distinct $a,b\in V$}.
$$
All other cases are then dealt with inductively in the proof of \cite[Thm. 6.33]{devriendt_2022_thesis}.
\end{proof}
\begin{conjecture}\label{conj: submodularity implies nonnegativity}
Let $G$ and $c$ be normalized with $\tau_\bullet(c)$ submodular. Then $p(c)\geq 0$.
\end{conjecture}
The author proved the partial result ``submodular $\sigma^2$ implies $p(c)\geq-1$" in \cite[Prop. 6.34]{devriendt_2022_thesis}. Conjecture \ref{conj: submodularity implies nonnegativity} is equivalent to the following conjecture on resistance nonnegative graphs:
\begin{conjecture}\label{conj: RN from submodularity}
A graph $G$ is RN if and only if $\mathcal{S}(G)$ intersects the submodular cone in $[0,1]^{2^V}$.
\end{conjecture}

\begin{example} 
We consider the diamond graph $G_{\hdiamond}$ with normalized constant weights $c\approx 0.531$. Figure \ref{fig: tau function} shows the resistance capacity $\tau_U(c)$ of the $6$ distinct subsets $U\subseteq V(G_{\hdiamond})$ of size $\vert U\vert\geq 2$, up to automorphisms. The resistance capacity is both monotonous and submodular.

\begin{figure}[h!]
    \centering
    \includegraphics[width=0.7\linewidth]{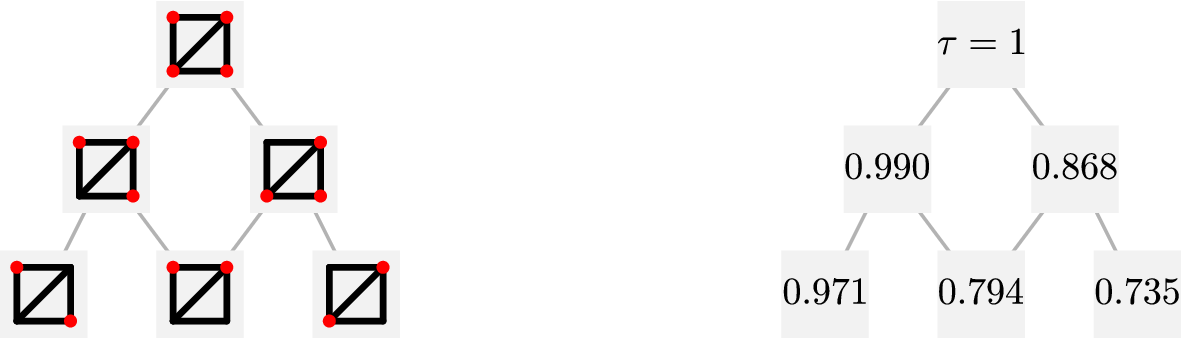}
    \caption{The resistance capacity of subsets of the vertices of $G_{\hdiamond}$ (rounded to 3-digit precision).}
    \label{fig: tau function}
\end{figure}
\end{example}
\section*{Acknowledgements} We thank Harry Richman for pointing out mistakes in an earlier version of the manuscript, Gwena\"{e}l Joret and Jean Cardinal for pointing to extended formulations of spanning tree polytopes (Remark \ref{rem: computation}) and an anonymous referee for suggestions that improved the presentation.


\bibliographystyle{abbrv}
\bibliography{bibliography.bib}

\bigskip
\bigskip
		
\noindent {\bf Authors' address:}
	
\smallskip
		
\noindent Karel Devriendt, University of Oxford
\hfill \url{karel.devriendt@maths.ox.ac.uk}

\end{document}